\documentclass[11pt,a4paper,reqno]{amsart}
\usepackage{amsfonts}
\usepackage{graphicx}   
\usepackage{graphics}
\usepackage{epsfig}
\usepackage{amssymb}
\usepackage{amsmath}
\usepackage{anysize}
\marginsize{2.5cm}{2.5cm}{2.5cm}{2.5cm}

\newtheorem{theorem}{Theorem}[section]

\newtheorem{remark}{Remark}[section]

\newcounter{Ex} 
\newcounter{Example-ENM}
\newcounter{Example-A} 
\newcounter{Example-B} 

\def\({\left(}
\def\){\right)}

\newlength\savedwidth



\begin{document}

\setcounter{page}{0}
\begin{flushleft}
{\scriptsize TWMS J. Pure Appl. Math., V.XX, N.XX, 20XX, pp.XX-XX}
\end{flushleft}
\bigskip
\bigskip
\title [A.N. Imomkulov: Classification of a family of ... ]
{Classification of a family of three dimensional real evolution algebras}

\author[TWMS J. Pure Appl. Math., V.XX, N.XX, 20XX] {A.N. Imomkulov$^{1}$}
\thanks{$^1${Institute of Mathematics},
 {Tashkent, Uzbekistan},\\
 \indent \,\,\, e-mail: aimomkulov@gmail.com\\
\indent \,\, \em Manuscript received XXXXXXXX 20XX}

\begin{abstract}
In this paper we classify a family of three-dimensional
real evolution algebras. We also consider an evolution
operator for an evolution algebra and find fixed
points of this operator for two and three-dimensional cases. Then we
construct an evolution algebra, the matrix of structural
constants of which is Jacobian of the evolution operator at a
fixed point. We study isomorphism
between these evolution algebras.


\end{abstract}
\maketitle

\smallskip
\section{Introduction} \label{sec:intro}

To study a non-linear function one usually finds the linear approximation 
to the function at a given point. The linear approximation of a function 
is the first order Taylor expansion around the point of interest.
In the theory of dynamical systems, linearization is a method for 
assessing the local stability of an equilibrium point. For an algebra 
(with a fixed multiplication $\ast$) and the (cubic) matrix $\mathcal M$ of
structural constants one can define a quadratic (non-linear) operator
 $F(x)=x\ast x$, with coefficients given by the matrix $\mathcal M$.
The Jacobian $J$ of $F$ at a given point, can be considered as a linear 
approximation of $F$. Consequently, $J$ generates an evolution algebra 
as its matrix of structural constants. 

Let $(A,\cdot)$ be an algebra over a field $K$. If it admits a countable basis
 $e_{1},e_{2},\dots,e_{n},\dots$, such that
 $$\begin{array}{cc}
 e_{i}\cdot e_{j}=0, \,\, if \,\,  i\neq j\\[2mm]
 e_{i}\cdot e_{i}=\sum_{k} a_{ik}e_{k}, \,\, for\,\, any\,\, i
\end{array}$$
then it is called an evolution algebra. This basis is called a natural basis.

We note that to every evolution algebra corresponds a square matrix
$(a_{ik})$ of structural constants of the given evolution algebra.\\
In \cite{10} the following basic properties of evolution algebras are proved:\\
1) Evolution algebras are not associative, in general.\\
2) Evolution algebras are commutative, flexible.\\
3) Evolution algebras are not power-associative, in general.\\
4) The direct sum of evolution algebras is also an evolution algebra.

In \cite{7} the dynamics of absolutely
nilpotent and idempotent elements in chains generated by two-dimensional
evolution algebras are studied. In \cite{2} the authors consider
an evolution algebra which has a rectangular matrix of structural constants.
This algebra is called evolution algebras of  \,
\textquotedblleft chicken\textquotedblright\,  population (EACP).
The mentioned paper is devoted to the description of structure of EACPs. Using the
Jordan form of the rectangular matrix of structural constants, a simple
description of EACPs over the field of complex numbers is given.
The classification of three-dimensional complex EACPs is obtained.
Moreover, some $(n + 1)$-dimensional EACPs are described. The fundamentals of 
evolution algebras have been being developed in the last years with no 
probabilistic restrictions on the stucture constants \cite{4,5,8}.

In Section \ref{sec:two} we study an approximation of two-dimensional
real evolution algebras and isomorphism
between these algebras. In Section \ref{section4} we will classify
a family of three-dimensional real evolution algebras.
We show that there are 13 class of such evolution algebras.
We consider an approximation of three-dimensional real evolution
algebras in Section \ref{section:5} and also study
isomorphism between these algebras.

\section{Approximation of two-dimensional real evolution algebras}\label{sec:two}

Let $E$ be a $2$-dimensional evolution algebra over the field of real numbers.
Such algebras are classified in \cite{6}:
\begin{theorem}\cite{6}\label{thmurodov}
Any two-dimensional real evolution algebra $E$ is isomorphic to one of the following
pairwise non-isomorphic algebras:

(i) $dim(E^{2})=1$:

$E_{1}$: $e_{1}e_{1}=e_{1}$, $e_{2}e_{2}=0$;

$E_{2}$: $e_{1}e_{1}=e_{1}$, $e_{2}e_{2}=e_{1}$;

$E_{3}$: $e_{1}e_{1}=e_{1}+e_{2}$, $e_{2}e_{2}=-e_{1}-e_{2}$;

$E_{4}$: $e_{1}e_{1}=e_{2}$, $e_{2}e_{2}=0$;

$E_{5}$: $e_{1}e_{1}=e_{2}$, $e_{2}e_{2}=-e_{2}$;

(ii) $dim(E^{2})=2$:

$E_{6}(a_{2};a_{3})$: $e_{1}e_{1}=e_{1}+a_{2}e_{2}$, $e_{2}e_{2}=a_{3}e_{1}+e_{2}$;
$1-a_{2}a_{3}\neq0$, $a_{2},a_{3}\in\mathbb{R}$. Moreover $E_{6}(a_{2};a_{3})$
is isomorphic to  $E_{6}(a_{3};a_{2})$.

$E_{7}(a_{4})$: $e_{1}e_{1}=e_{2}$, $e_{2}e_{2}=e_{1}+a_{4}e_{4}$, where $a_{4}\in\mathbb{R}$;
\end{theorem}
For a given evolution algebra $(E,\cdot)$
an evolution operator has the following form $F(x)=x\cdot x=x^2$.
If $x=\sum_{i=1}^n x_{i}e_{i}$ then
 $$x^2=\sum_{i=1}^n x_{i}^2 e_{i}^2=\sum_{i=1}^n x_{i}^2\big(\sum_{k=1}^n a_{ik} e_{k}\big)=
 \sum_{k=1}^n\big(\sum_{i=1}^n a_{ik} x_{i}^2\big)e_{k}.$$
 We denote $x_{k}'=\sum_{i=1}^n a_{ik} x_{i}^2$. Thus we have the following operator, $F:E\rightarrow E$,
 $$F:x_{k}'=\sum_{i=1}^n a_{ik} x_{i}^2, k=\overline{1,n}.$$

Jacobian of the operator $F$ at the point $x$ for two-dimensional case has a form
\begin{displaymath}
J_{F}(x)=\left( \begin{array}{cccc}
2a_{11}x_{1} & 2a_{21}x_{2} \\
2a_{12}x_{1} & 2a_{22}x_{2} \\
\end{array} \right).
\end{displaymath}

Following \cite{9} and \cite{3} we define an evolution algebra $\widetilde E$ with matrix
$J_{F}(x)$ as the matrix of structural constants.\\

We will find fixed points of this operator, i.e. solutions of $F(x)=x:$
\begin{equation}
\label{fix:point}
\left\{ \begin{array}{ll}
x_{1}=a_{11}x_{1}^2+a_{21}x_{2}^2, \\[2mm]
x_{2}=a_{12}x_{1}^2+a_{22}x_{2}^2.
\end{array} \right.
\end{equation}
Note that $(0,0)$ is one of solutions of system of equations (\ref{fix:point}), and
\begin{displaymath}
J_{F}(0,0)=\left( \begin{array}{cccc}
0 & 0 \\
0 & 0 \\
\end{array} \right).
\end{displaymath}
So corresponding evolution algebra with matrix $J_{F}(0,0)$ is trivial.

Trivial evolution algebras are not interesting. So we will find a
non-zero solutions denoted by $(x_{1}^0;x_{2}^0)$ of (\ref{fix:point})
for algebras $E_{i}, \,\, i=1,2,\dots,7$ mentioned in Theorem \ref{thmurodov} 
and study isomorphisms of evolution algebras
corresponding to these fixed points with other evolution algebras.

In the following table we give all possibilities for two-dimensional case:\\

\begin{tabular}{l|c|l}
\hline
2-dimensional real & Non-zero real fixed& Corresponding evolution\\
evolution algebras& points of the operator $F$& algebras to fixed points\\
\hline
$E_{1}:\left(\begin{array}{cccc}
1 & 0\\
0 & 0\end{array}\right)$ &
$(1;0)$ &
$\widetilde E_{1}:\left(\begin{array}{cccc}
2 & 0\\
0 & 0\end{array}\right)$\\[3mm]
$E_{2}:\left(\begin{array}{cccc}
1 & 0\\
1 & 0\end{array}\right)$ &
$(1;0)$ &
$\widetilde E_{2}:\left(\begin{array}{cccc}
2 & 0\\
0 & 0\end{array}\right)$\\[3mm]
$E_{3}:\left(\begin{array}{cccc}
1 & 1\\
-1 & -1\end{array}\right)$ &
Not exists &
Not exists\\[3mm]
$E_{4}:\left(\begin{array}{cccc}
0 & 1\\
0 & 0\end{array}\right)$ &
Not exists &
Not exists\\[3mm]
$E_{5}:\left(\begin{array}{cccc}
0 & 1\\
0 & -1\end{array}\right)$ &
$(0;-1)$ &
$\widetilde E_{5}:\left(\begin{array}{cccc}
0 & 0\\
0 & 2\end{array}\right)$\\[3mm]
$E_{6}(a_{2};a_{3}):\left(\begin{array}{cccc}
1 & a_{2}\\
a_{3} & 1\end{array}\right)$ &
$(x_{1}^0;x_{2}^0)$ &
$\widetilde E_{6}(a_{2};a_{3}):\left(\begin{array}{cccc}
2x_{1}^0 & 2a_{3}x_{2}^0\\
2a_{2}x_{1}^0 & 2x_{2}^0\end{array}\right)$\\[3mm]
$E_{7}(a_{4}):\left(\begin{array}{cccc}
0 & 1\\
1 & a_{4}\end{array}\right)$ &
$(x_{1}^0;x_{2}^0)$ if $a_{4}\geq-\frac{3}{\sqrt[3]{4}}$ &
$\widetilde E_{7}(a_{4}):\left(\begin{array}{cccc}
0 & 2x_{2}^0\\
2x_{1}^0 & 2a_{4}x_{2}^0\end{array}\right)$\\
\hline
\end{tabular}\\

We have the following theorem.
\begin{theorem}
$i)$ Evolution algebras $\widetilde E_{1}$, $\widetilde E_{2}$ and $\widetilde E_{5}$ are
isomorphic to $E_{1}$; \\
$ii)$ $\widetilde E_{6}(a_{2};a_{3})$ is isomorphic to the evolution
algebra $E_{6}(b_{2};b_{3})$, where $b_{2}=a_{3}\left(\frac{x_{2}^0}{x_{1}^0}\right)^2$,
\,\,\,\,\,\,\,\,\,\,\,\,\,\,\,\,\,\,\,\,\,\,\,\
$b_{3}=a_{2}\left(\frac{x_{1}^0}{x_{2}^0}\right)^2$;\\
$iii)$ $\widetilde E_{7}(a_{4})$ is isomorphic to $E_{7}(b_{4})$, 
where $b_{4}=a_{4}\sqrt[3]{(\frac{x_{2}^0}{x_{1}^0})^2}$.
\end{theorem}

\begin{proof}
$\widetilde E_{1}\approxeq E_{1}$: By the change of basis $\widetilde e_{1}=\frac{1}{2}e_{1}$
we can prove that the evolution algebra $\widetilde E_{1}$ is isomorphic to $E_{1}$.\\
$\widetilde E_{2}\approxeq E_{1}$: It is similar to the above proof.\\
$\widetilde E_{5}\approxeq E_{1}$: By the change of basis
$\widetilde e_{1}=\frac{1}{2}e_{2}$, $\widetilde e_{2}=e_{1}$ we can
prove that the evolution algebra $\widetilde E_{5}$ is isomorphic to $E_{1}$.\\
$\widetilde E_{6}(a_{2};a_{3})\approxeq E_{6}(b_{2};b_{3})$:
We can see this by the change of basis $\widetilde e_{1}=\frac{1}{2x_{1}^0}e_{1}$ and
$\widetilde e_{2}=\frac{1}{2x_{2}^0}e_{2}$.\\
$\widetilde E_{7}(a_{4})\approxeq E_{7}(b_{4})$:
We can see this by the change of basis $\widetilde e_{1}=2\sqrt[3]{x_{1}^0(x_{2}^0)^2}e_{1}$ and
$\widetilde e_{2}=2\sqrt[3]{(x_{1}^0)^2x_{2}^0}e_{2}$.
\end{proof}

\section{Three-dimensional real evolution algebras with $dim(E^2)=1$}\label{section4}

In \cite{1} three dimensional complex evolution algebras are classified.
Now we shall consider classification of three dimensional real evolution algebras.

Fix a three-dimensional real evolution algebra $E$ and a natural
basis $B=\{e_{1},e_{2},e_{3}\}$. Let $M_{B}$ be the matrix of 
structural constants of $E$ relative to $B$:
$$M_{B}=\left(\begin{array}{cccccc}
a_{11} & a_{12} & a_{13}\\
a_{21} & a_{22} & a_{23}\\
a_{31} & a_{32} & a_{33}
\end{array}\right).$$

In order to classify three dimensional real evolution algebras
with condition $dim(E^2)=1$ we find a basis of $E$ for which
its structure matrix has an expression as simple as possible, where by
'simple' we mean with the maximal number of $0$, $1$ and $-1$ in the entries.

Let $dim(E^2)=1$. Without loss of generality we may assume $e_{1}^2\neq 0$.
Write $e_{1}^2=a_{1}e_{1}+a_{2}e_{2}+a_{3}e_{3}$, where $a_{i}\in\mathbb{R}$
and $a_{i}\neq 0$ for some $i$. Note that ${e_{1}^2}$ is basis of $E^2$.

Since $e_{2}^2,e_{3}^2\in E^2$, there exist $c_{1},c_{2}\in\mathbb{R}$ such that
$$e_{2}^2=c_{1}e_{1}^2=c_{1}(a_{1}e_{1}+a_{2}e_{2}+a_{3}e_{3}),$$
$$e_{3}^2=c_{2}e_{1}^2=c_{2}(a_{1}e_{1}+a_{2}e_{2}+a_{3}e_{3}).$$

Then $$M_{B}=\left(\begin{array}{cccccc}
a_{1} & a_{2} & a_{3}\\
c_{1}a_{1} & c_{1}a_{2} & c_{1}a_{3}\\
c_{2}a_{1} & c_{2}a_{2} & c_{2}a_{3}
\end{array}\right).$$

We analyze when $E^2$ has the extension property. This means that there
exists a natural basis $B'=\{e_{1}',e_{2}',e_{3}'\}$ of $E$ with
\begin{equation}\begin{array}{cc}
e_{1}'=e_{1}^2=a_{1}e_{1}+a_{2}e_{2}+a_{3}e_{3} \\
e_{2}'=\alpha e_{1}+\beta e_{2}+\gamma e_{3} \\
e_{3}'=\delta e_{1}+\nu e_{2}+\eta e_{3}
\end{array}
\end{equation}
for some $\alpha,\beta,\gamma,\delta,\nu,\eta\in\mathbb{R}$
such that $\beta\eta-\gamma\nu\neq 0$.
This implies that
\begin{equation}\label{detB}
\left|P_{B'B}\right|=\left|\begin{array}{cccccc}
a_{1} & a_{2} & a_{3}\\
\alpha & \beta & \gamma\\
\delta & \nu & \eta
\end{array}\right|\neq 0.
\end{equation}

By products $e_{1}'e_{2}'=0, e_{1}'e_{3}'=0, e_{2}'e_{3}'=0$,
$B'$ is a natural basis if and only if the following conditions are satisfied:
\begin{equation}\label{alphauchun}
\alpha a_{1}+\beta a_{2}c_{1}+\gamma a_{3}c_{2}=0
\end{equation}
\begin{equation}\label{deltauchun}
\delta a_{1}+\nu a_{2}c_{1}+\eta a_{3}c_{2}=0
\end{equation}
$$\alpha\delta+\beta\nu c_{1}+\gamma\eta c_{2}=0.$$

In the above conditions, the structure matrix of $E$ relative to $B'$ is:
$$M_{B'}=\left(\begin{array}{cccccc}
a_{1}^2+a_{2}^2c_{1}+a_{3}^2c_{2} & 0 & 0\\
\alpha^2+\beta^2c_{1}+\gamma^2c_{2} & 0 & 0\\
\delta^2+\nu^2c_{1}+\eta^2c_{2} & 0 & 0
\end{array}\right).$$

Now, we start with the analysis of possible cases.\\[2mm]
\textbf{Case 1.} Suppose that $a_{1}\neq 0$.\\
By changing the basis, we may assume that $e_{1}^2=e_{1}+a_{2}e_{2}+a_{3}e_{3}$.
Using (\ref{alphauchun}) we get $\alpha=-(\beta a_{2}c_{1}+\gamma a_{3}c_{2})$
and by (\ref{deltauchun}), $\delta=-(\nu a_{2}c_{1}+\eta a_{3}c_{2})$.
If we replace $\alpha$ and $\delta$ in (\ref{detB}) we obtain that:
$$\left|P_{B'B}\right|=(1+a_{2}^2c_{1}+a_{3}^2c_{2})(\beta\eta-\gamma\nu).$$

Now we check that $|P_{B'B}|$ is zero or not.
This happens depending on $1+a_{2}^2c_{1}+a_{3}^2c_{2}$ being zero or not.\\[2mm]
\textbf{Case 1.1} Assume $1+a_{2}^2c_{1}+a_{3}^2c_{2}=0$.\\
In this case $E^2$ has not the extension property since $|P_{B'B}|=0$.
We will analyze what happens when $1+a_{3}^2c_{2}\neq 0$ and when $1+a_{3}^2c_{2}=0$.\\[2mm]
\textbf{Case 1.1.1} If $1+a_{3}^2c_{2}\neq 0$.\\
Note that $a_{2}^2c_{1}\neq 0$ since otherwise we get a contradiction.
Then $c_{1}=\frac{-1-a_{3}^2c_{2}}{a_{2}^2}$. In this case, the structure matrix is:
$$M_{B}=\left(\begin{array}{cccccc}
1 & a_{2} & a_{3}\\[2mm]
\frac{-1-a_{3}^2c_{2}}{a_{2}^2} & \frac{-1-a_{3}^2c_{2}}{a_{2}}
& \frac{(-1-a_{3}^2c_{2})a_{3}}{a_{2}^2}\\[2mm]
c_{2} & c_{2}a_{2} & c_{2}a_{3}\end{array}\right).$$
\textbf{Case 1.1.1.1} Suppose that $a_{3}\neq 0$.

If we take the natural basis $B''=\{e_{1},a_{2}e_{2},a_{3}e_{3}\}$, then
\begin{equation}\label{yangibasis}
M_{B''}=\left(\begin{array}{cccccc}
1 & 1 & 1\\[2mm]
-1-a_{3}^2c_{2} & -1-a_{3}^2c_{2} & -1-a_{3}^2c_{2}\\[2mm]
a_{3}^2c_{2} & a_{3}^2c_{2} & a_{3}^2c_{2}\end{array}\right).
\end{equation}

We are going to verify two cases: $c_{2}=0$ and $c_{2}\neq 0$.

Assume first $c_{2}=0$. Then
$M_{B''}=\left(\begin{array}{cccccc}
1 & 1 & 1\\
-1 & -1 & -1\\
0 & 0 & 0 \end{array}\right)$.
By considering another change of basis we find a structure matrix
with more zeros. Namely, let $B'''=\{e_{2},e_{1}+e_{3},e_{3}\}$.
Then
$$M_{B'''}=\left(\begin{array}{cccccc}
1 & 1 & 0\\
-1 & -1 & 0\\
0 & 0 & 0 \end{array}\right).$$

In what follows we will assume that $c_{2}\neq 0$. We recall that
we are considering the structure matrix given in (\ref{yangibasis}).
Take $I:=<(1+a_{3}^2c_{2})e_{1}+e_{2}>$. Then $I$ is a two-dimensional
evolution ideal which is degenerate as an evolution algebra.

Now, for $B'''$ the natural basis change is
$$
P_{B'''B''}=\left(\begin{array}{cccccc}
\frac{1+D}{2(1+a_{3}^2c_{2})}
& \frac{-1+D}{2(1+a_{3}^2c_{2})}
& \frac{1+D}{2(1+a_{3}^2c_{2})}\\[2mm]
\frac{-1+D}{2(1+a_{3}^2c_{2})}
& \frac{1+D}{2(1+a_{3}^2c_{2})}
& \frac{-1+D}{2(1+a_{3}^2c_{2})}\\[2mm]
-(a_{3}^2c_{2}) & 0 & 1 \end{array}\right)$$
where $D=(a_{3}^2c_{2})^3+2(a_{3}^2c_{2})^2+(a_{3}^2c_{2})$
and we obtain: $$M_{B'''}=\left(\begin{array}{cccccc}
1 & 1 & 0\\
-1 & -1 & 0\\
1 & 1 & 0 \end{array}\right).$$
Note that $|P_{B'''B''}|=-2(a_{3}^2c_{2})(1+a_{3}^2c_{2})^2\neq 0$
because $a_{3}^2c_{2}\neq 0$ and $a_{3}^2c_{2}\neq-1.$\\[2mm]
\textbf{Case 1.1.1.2} Suppose that $a_{3}=0$.
Then $1+a_{2}^2c_{1}=0$ and necessarily $a_{2}^2c_{1}\neq 0$. In this case,
\begin{equation}
M_{B}=\left(\begin{array}{cccccc}
1 & a_{2} & 0\\[2mm]
\frac{-1}{a_{2}^2} & \frac{-1}{a_{2}} & 0\\[2mm]
c_{2} & c_{2}a_{2} & 0\end{array}\right).
\end{equation}

Again we will verify two cases depending on $c_{2}$.

Assume $c_{2}>0$. Take $B''=\{e_{1},a_{2}e_{2},\frac{1}{\sqrt{c_{2}}}e_{3}\}$.
Then
 $M_{B''}=\left(\begin{array}{cccccc}
1 & 1 & 0\\
-1 & -1 & 0\\
1 & 1 & 0 \end{array}\right)$, which has already appeared.

Assume $c_{2}<0$. Take $B''=\{e_{1},a_{2}e_{2},\frac{1}{\sqrt{-c_{2}}}e_{3}\}$.
Then
$M_{B''}=\left(\begin{array}{cccccc}
1 & 1 & 0\\
-1 & -1 & 0\\
-1 & -1 & 0 \end{array}\right)$.

Suppose $c_{2}=0$. Then, for $B''=\{e_{1},a_{2}e_{2},e_{3}\}$ we
have
$M_{B''}=\left(\begin{array}{cccccc}
1 & 1 & 0\\
-1 & -1 & 0\\
0 & 0 & 0 \end{array}\right)$, matrix that has already appeared.\\[2mm]
\textbf{Case 1.1.2} Suppose that $1+a_{3}^2c_{2}=0$.\\
This implies that $a_{3}^2c_{2}\neq 0$ and $a_{2}^2c_{1}=0$.\\[2mm]
\textbf{Case 1.1.2.1} Assume $c_{1}\neq 0$.\\
This implies that $a_{2}=0$. Moreover, since $a_{3}\neq 0$, necessarily
$c_{2}=\frac{-1}{a_{3}^2}$. If we take the natural basis $B''=\{e_{1},e_{3},e_{2}\}$,
then
$M_{B''}=\left(\begin{array}{cccccc}
1 & a_{3} & 0\\[2mm]
\frac{-1}{a_{3}^2} & \frac{-1}{a_{3}} & 0\\[2mm]
0 & 0 & 0 \end{array}\right)$ and we are in the Case 1.1.1.2.\\[2mm]
\textbf{Case 1.1.2.2} Suppose $c_{1}=0$ and $a_{2}=0$.\\
Take $B''=\{e_{1},a_{3}e_{3},e_{2}\}$. Then
$M_{B''}=\left(\begin{array}{cccccc}
1 & 1 & 0\\
-1 & -1 & 0\\
0 & 0 & 0\end{array}\right)$ as above.\\[2mm]
\textbf{Case 1.1.2.3} Suppose $c_{1}=0$ and $a_{2}\neq 0$.\\
Taking $B''=\{e_{1},e_{3},e_{2}\}$, we are in the same conditions as
in the Case 1.1.1.1 with $c_{2}=0$.\\[2mm]
\textbf{Case 1.2} Assume $1+a_{2}^2c_{1}+a_{3}^2c_{2}\neq 0.$\\
We will prove that $E^2$ has the extension property. In any subcase
we will provide with a natural basis for $E$ one of its elements
gives a natural basis of $E^2$.\\[2mm]
\textbf{Case 1.2.1} Suppose that $c_{1}=c_{2}=0$.\\
Consider the natural basis $B'=\{e_{1}^2,e_{2}+e_{3},2e_{2}+e_{3}\}$.
Then
$$M_{B'}=\left(\begin{array}{cccccc}
1 & 0 & 0\\
0 & 0 & 0\\
0 & 0 & 0\end{array}\right).$$
We claim that this evolution algebra does not have a two-dimensional
evolution ideal generated by one element. To prove this consider
$f=me_{1}+ne_{2}+pe_{3}$. Then the ideal $I$ generated by $f$ is
the linear span of $\{f\}\cup\{m^ie_{i}\}_{i\in\mathbb{N}}.$
In order for $I$ to have a natural basis with two elements,
necessarily $m=0$, implying that the dimension of $I$ is one, a contradiction.\\[2mm]
\textbf{Case 1.2.2} Assume that $c_{1}=0$ and ${c_{2}\neq 0}$.\\
Then $1+c_{2}a_{3}^2\neq 0.$ For
$B'=\{e_{1}+a_{2}e_{2}+a_{3}e_{3}, e_{2}, -a_{3}c_{2}e_{1}+e_{2}+e_{3}\}$
the structure matrix is
$$M_{B'}=\left(\begin{array}{cccccc}
1+c_{2}a_{3}^2 & 0 & 0\\
0 & 0 & 0\\
c_{2}(1+c_{2}a_{3}^2) & 0 & 0\end{array}\right).$$

Note that $E^2$ has the extension property because the first
element in $B'$ is $e_{1}^2$, which is a natural basis of $E^2$.\\[2mm]
\textbf{Case 1.2.2.1} Assume that $c_{2}>0$.\\
Consider $B''=\left\{\frac{1}{1+c_{2}a_{3}^2}e_{1},e_{2},
\frac{1}{\sqrt{c_{2}}(1+c_{2}a_{3}^2)}e_{3}\right\}.$ Then
$$M_{B''}=\left(\begin{array}{cccccc}
1 & 0 & 0\\
0 & 0 & 0\\
1 & 0 & 0\end{array}\right).$$
\textbf{Case 1.2.2.2} Assume that $c_{2}<0$.\\
Consider $B''=\left\{\frac{1}{1+c_{2}a_{3}^2}e_{1},e_{2},
\frac{1}{\sqrt{-c_{2}}(1+c_{2}a_{3}^2)}e_{3}\right\}.$ Then
$$M_{B''}=\left(\begin{array}{cccccc}
1 & 0 & 0\\
0 & 0 & 0\\
-1 & 0 & 0\end{array}\right).$$

We claim that these evolution algebras do not have a
two-dimensional evolution ideal generated by one element.
Let $f=\alpha e_{1}+\beta e_{2}+\gamma e_{3}$. Then the ideal
generated by $f$, say $I$,  is the linear span of $\{f,\gamma e_{1},\alpha e_{1}\}\cup\{(\alpha^2+\gamma^2)\alpha^ie_{1}\}_{i\in\mathbb{N}\cup\{0\}}
\cup\{(\alpha^2+\gamma^2)^2\alpha^ie_{1}\}_{i\in\mathbb{N}\cup\{0\}}.$
After some computations, in order for $I$ to have dimension 2 and to
be degenerated implies $\alpha=0$ or $\gamma=0$, a contradiction.\\[2mm]
\textbf{Case 1.2.3} If $c_{1}>0$ and $c_{2}>0.$\\
If $B'$ is the natural basis such that $P_{B'B}=\left(\begin{array}{cccccc}
1 & a_{2} & a_{3}\\[2mm]
-a_{2}c_{1} & 1 & 0\\[2mm]
\frac{-a_{3}c_{2}}{1+c_{1}a_{2}^2} & \frac{-a_{3}a_{2}c_{2}}{1+c_{1}a_{2}^2} & 1\end{array}\right)$,
we obtain that $M_{B'}=\left(\begin{array}{cccccc}
1+a_{2}^2c_{1}+a_{3}^2c_{2} & 0 & 0\\[2mm]
c_{1}(1+c_{1}a_{2}^2) & 0 & 0\\[2mm]
\frac{c_{2}(1+a_{2}^2c_{1}+a_{3}^2c_{2})}{(1+c_{1}a_{2}^2)} & 0 & 0\end{array}\right).$\\
Now, consider the natural basis $B''=\{f_{1},f_{2},f_{3}\}$
such that $$P_{B''B'}=\left(\begin{array}{cccccc}
\frac{1}{1+a_{2}^2c_{1}+a_{3}^2c_{2}} & 0 & 0\\[2mm]
0 & \frac{1}{\sqrt{c_{1}(1+c_{1}a_{2}^2)(1+a_{2}^2c_{1}+a_{3}^2c_{2})}} & 0\\[2mm]
0 & 0 & \frac{\sqrt{1+c_{1}a_{2}^2}}{\sqrt{c_{2}}(1+a_{2}^2c_{1}+a_{3}^2c_{2})}\end{array}\right)$$
and the structure matrix is
$$M_{B''}=\left(\begin{array}{cccccc}
1 & 0 & 0\\
1 & 0 & 0\\
1 & 0 & 0\end{array}\right).$$
It is not difficult to show that this evolution algebra does
not have a degenerate two-dimensional evolution ideal generated by one element.\\[2mm]
\textbf{Case 1.2.4} If $c_{1}>0$ and $c_{2}<0.$\\
For the natural basis $B'$ such that $P_{B'B}=\left(\begin{array}{cccccc}
1 & a_{2} & a_{3}\\[2mm]
-a_{2}c_{1} & 1 & 0\\[2mm]
\frac{-a_{3}c_{2}}{1+c_{1}a_{2}^2} & \frac{-a_{3}a_{2}c_{2}}{1+c_{1}a_{2}^2} & 1\end{array}\right)$,
we obtain that $M_{B'}=\left(\begin{array}{cccccc}
1+a_{2}^2c_{1}+a_{3}^2c_{2} & 0 & 0\\[2mm]
c_{1}(1+c_{1}a_{2}^2) & 0 & 0\\[2mm]
\frac{c_{2}(1+a_{2}^2c_{1}+a_{3}^2c_{2})}{(1+c_{1}a_{2}^2)} & 0 & 0\end{array}\right).$\\[2mm]
\textbf{Case 1.2.4.1} Assume $1+a_{2}^2c_{1}+a_{3}^2c_{2}>0.$\\
Now, consider the natural basis $B''=\{f_{1},f_{2},f_{3}\}$
such that $$P_{B''B'}=\left(\begin{array}{cccccc}
\frac{1}{1+a_{2}^2c_{1}+a_{3}^2c_{2}} & 0 & 0\\[2mm]
0 & \frac{1}{\sqrt{c_{1}(1+c_{1}a_{2}^2)(1+a_{2}^2c_{1}+a_{3}^2c_{2})}} & 0\\[2mm]
0 & 0 & \frac{\sqrt{1+c_{1}a_{2}^2}}{\sqrt{-c_{2}}(1+a_{2}^2c_{1}+a_{3}^2c_{2})}\end{array}\right)$$
and the structure matrix is
$$M_{B''}=\left(\begin{array}{cccccc}
1 & 0 & 0\\
1 & 0 & 0\\
-1 & 0 & 0\end{array}\right).$$\\[2mm]
\textbf{Case 1.2.4.2} Assume $1+a_{2}^2c_{1}+a_{3}^2c_{2}<0.$\\
Now, consider the natural basis $B''=\{f_{1},f_{2},f_{3}\}$
such that $$P_{B''B'}=\left(\begin{array}{cccccc}
\frac{1}{1+a_{2}^2c_{1}+a_{3}^2c_{2}} & 0 & 0\\[2mm]
0 & \frac{1}{\sqrt{-c_{1}(1+c_{1}a_{2}^2)(1+a_{2}^2c_{1}+a_{3}^2c_{2})}} & 0\\[2mm]
0 & 0 & \frac{\sqrt{1+c_{1}a_{2}^2}}{\sqrt{-c_{2}}(1+a_{2}^2c_{1}+a_{3}^2c_{2})}\end{array}\right)$$
and the structure matrix is:
$$M_{B''}=\left(\begin{array}{cccccc}
1 & 0 & 0\\
-1 & 0 & 0\\
-1 & 0 & 0\end{array}\right).$$\\[2mm]
\textbf{Case 1.2.5} If $c_{1}<0$ and $c_{2}>0.$\\[2mm]
\textbf{Case 1.2.5.1} Assume $1+a_{2}^2c_{1}>0.$\\
For the natural basis $B'$ such that $P_{B'B}=\left(\begin{array}{cccccc}
1 & a_{2} & a_{3}\\[2mm]
-a_{2}c_{1} & 1 & 0\\[2mm]
\frac{-a_{3}c_{2}}{1+c_{1}a_{2}^2} & \frac{-a_{3}a_{2}c_{2}}{1+c_{1}a_{2}^2} & 1\end{array}\right)$,
we obtain that $M_{B'}=\left(\begin{array}{cccccc}
1+a_{2}^2c_{1}+a_{3}^2c_{2} & 0 & 0\\[2mm]
c_{1}(1+c_{1}a_{2}^2) & 0 & 0\\[2mm]
\frac{c_{2}(1+a_{2}^2c_{1}+a_{3}^2c_{2})}{(1+c_{1}a_{2}^2)} & 0 & 0\end{array}\right).$\\[2mm]
\textbf{Case 1.2.5.1.1} Assume $1+a_{2}^2c_{1}+a_{3}^2c_{2}>0.$\\
Now, consider the natural basis $B''=\{f_{1},f_{2},f_{3}\}$
such that $$P_{B''B'}=\left(\begin{array}{cccccc}
\frac{1}{1+a_{2}^2c_{1}+a_{3}^2c_{2}} & 0 & 0\\[2mm]
0 & \frac{1}{\sqrt{-c_{1}(1+c_{1}a_{2}^2)(1+a_{2}^2c_{1}+a_{3}^2c_{2})}} & 0\\[2mm]
0 & 0 & \frac{\sqrt{1+c_{1}a_{2}^2}}{\sqrt{c_{2}}(1+a_{2}^2c_{1}+a_{3}^2c_{2})}\end{array}\right)$$
and the structure matrix is:
$$M_{B''}=\left(\begin{array}{cccccc}
1 & 0 & 0\\
-1 & 0 & 0\\
1 & 0 & 0\end{array}\right).$$
It is not difficult to show that this evolution algebra
does not have a degenerate two-dimensional evolution ideal generated by one element.\\[2mm]
\textbf{Case 1.2.5.1.2} Assume $1+a_{2}^2c_{1}+a_{3}^2c_{2}<0.$\\
Now, consider the natural basis $B''=\{f_{1},f_{2},f_{3}\}$
such that $$P_{B''B'}=\left(\begin{array}{cccccc}
\frac{1}{1+a_{2}^2c_{1}+a_{3}^2c_{2}} & 0 & 0\\[2mm]
0 & \frac{1}{\sqrt{c_{1}(1+c_{1}a_{2}^2)(1+a_{2}^2c_{1}+a_{3}^2c_{2})}} & 0\\[2mm]
0 & 0 & \frac{\sqrt{1+c_{1}a_{2}^2}}{\sqrt{c_{2}}(1+a_{2}^2c_{1}+a_{3}^2c_{2})}\end{array}\right)$$
and the structure matrix is:
$M_{B''}=\left(\begin{array}{cccccc}
1 & 0 & 0\\
1 & 0 & 0\\
1 & 0 & 0\end{array}\right),$ which has already appeared.\\[2mm]
\textbf{Case 1.2.5.2} Assume $1+a_{2}^2c_{1}<0.$\\
For $B'$ the natural basis such that $P_{B'B}=\left(\begin{array}{cccccc}
1 & a_{2} & a_{3}\\[2mm]
-a_{2}c_{1} & 1 & 0\\[2mm]
\frac{-a_{3}c_{2}}{1+c_{1}a_{2}^2} & \frac{-a_{3}a_{2}c_{2}}{1+c_{1}a_{2}^2} & 1\end{array}\right)$,
we obtain that $M_{B'}=\left(\begin{array}{cccccc}
1+a_{2}^2c_{1}+a_{3}^2c_{2} & 0 & 0\\[2mm]
c_{1}(1+c_{1}a_{2}^2) & 0 & 0\\[2mm]
\frac{c_{2}(1+a_{2}^2c_{1}+a_{3}^2c_{2})}{(1+c_{1}a_{2}^2)} & 0 & 0\end{array}\right).$\\[2mm]
\textbf{Case 1.2.5.2.1} Assume $1+a_{2}^2c_{1}+a_{3}^2c_{2}>0.$\\
Now, consider the natural basis $B''=\{f_{1},f_{2},f_{3}\}$
such that $$P_{B''B'}=\left(\begin{array}{cccccc}
\frac{1}{1+a_{2}^2c_{1}+a_{3}^2c_{2}} & 0 & 0\\[2mm]
0 & \frac{1}{\sqrt{c_{1}(1+c_{1}a_{2}^2)(1+a_{2}^2c_{1}+a_{3}^2c_{2})}} & 0\\[2mm]
0 & 0 & \frac{\sqrt{-(1+c_{1}a_{2}^2)}}{\sqrt{c_{2}}(1+a_{2}^2c_{1}+a_{3}^2c_{2})}\end{array}\right)$$
and the structure matrix is:
$M_{B''}=\left(\begin{array}{cccccc}
1 & 0 & 0\\
1 & 0 & 0\\
-1 & 0 & 0\end{array}\right),$ which has already appeared.\\[2mm]
\textbf{Case 1.2.5.2.2} Assume $1+a_{2}^2c_{1}+a_{3}^2c_{2}<0.$\\
Now, consider the natural basis $B''=\{f_{1},f_{2},f_{3}\}$
such that $$P_{B''B'}=\left(\begin{array}{cccccc}
\frac{1}{1+a_{2}^2c_{1}+a_{3}^2c_{2}} & 0 & 0\\[2mm]
0 & \frac{1}{\sqrt{-c_{1}(1+c_{1}a_{2}^2)(1+a_{2}^2c_{1}+a_{3}^2c_{2})}} & 0\\[2mm]
0 & 0 & \frac{\sqrt{-(1+c_{1}a_{2}^2)}}{\sqrt{c_{2}}(1+a_{2}^2c_{1}+a_{3}^2c_{2})}\end{array}\right)$$
and the structure matrix is:
$M_{B''}=\left(\begin{array}{cccccc}
1 & 0 & 0\\
-1 & 0 & 0\\
-1 & 0 & 0\end{array}\right),$ which has already appeared.\\[2mm]
\textbf{Case 1.2.6} If $c_{1}<0$ and $c_{2}<0.$\\[2mm]
\textbf{Case 1.2.6.1} Assume $1+a_{2}^2c_{1}>0.$\\
For the natural basis $B'$ such that $P_{B'B}=\left(\begin{array}{cccccc}
1 & a_{2} & a_{3}\\[2mm]
-a_{2}c_{1} & 1 & 0\\[2mm]
\frac{-a_{3}c_{2}}{1+c_{1}a_{2}^2} & \frac{-a_{3}a_{2}c_{2}}{1+c_{1}a_{2}^2} & 1\end{array}\right)$,
we obtain that $M_{B'}=\left(\begin{array}{cccccc}
1+a_{2}^2c_{1}+a_{3}^2c_{2} & 0 & 0\\[2mm]
c_{1}(1+c_{1}a_{2}^2) & 0 & 0\\[2mm]
\frac{c_{2}(1+a_{2}^2c_{1}+a_{3}^2c_{2})}{(1+c_{1}a_{2}^2)} & 0 & 0\end{array}\right).$\\[2mm]
\textbf{Case 1.2.6.1.1} Assume $1+a_{2}^2c_{1}+a_{3}^2c_{2}>0.$\\
Now, consider the natural basis $B''=\{f_{1},f_{2},f_{3}\}$
such that $$P_{B''B'}=\left(\begin{array}{cccccc}
\frac{1}{1+a_{2}^2c_{1}+a_{3}^2c_{2}} & 0 & 0\\[2mm]
0 & \frac{1}{\sqrt{-c_{1}(1+c_{1}a_{2}^2)(1+a_{2}^2c_{1}+a_{3}^2c_{2})}} & 0\\[2mm]
0 & 0 & \frac{\sqrt{1+c_{1}a_{2}^2}}{\sqrt{-c_{2}}(1+a_{2}^2c_{1}+a_{3}^2c_{2})}\end{array}\right)$$
and the structure matrix is:
$M_{B''}=\left(\begin{array}{cccccc}
1 & 0 & 0\\
-1 & 0 & 0\\
-1 & 0 & 0\end{array}\right),$ which has appeared above.\\[2mm]
\textbf{Case 1.2.6.1.2} Assume $1+a_{2}^2c_{1}+a_{3}^2c_{2}<0.$\\
Now, consider the natural basis $B''=\{f_{1},f_{2},f_{3}\}$
such that $$P_{B''B'}=\left(\begin{array}{cccccc}
\frac{1}{1+a_{2}^2c_{1}+a_{3}^2c_{2}} & 0 & 0\\[2mm]
0 & \frac{1}{\sqrt{c_{1}(1+c_{1}a_{2}^2)(1+a_{2}^2c_{1}+a_{3}^2c_{2})}} & 0\\[2mm]
0 & 0 & \frac{\sqrt{1+c_{1}a_{2}^2}}{\sqrt{-c_{2}}(1+a_{2}^2c_{1}+a_{3}^2c_{2})}\end{array}\right)$$
and the structure matrix is:
$M_{B''}=\left(\begin{array}{cccccc}
1 & 0 & 0\\
1 & 0 & 0\\
-1 & 0 & 0\end{array}\right),$ which has already appeared.\\[2mm]
\textbf{Case 1.2.6.2} Assume $1+a_{2}^2c_{1}<0.$\\
For the natural basis $B'$ such that $P_{B'B}=\left(\begin{array}{cccccc}
1 & a_{2} & a_{3}\\[2mm]
-a_{2}c_{1} & 1 & 0\\[2mm]
\frac{-a_{3}c_{2}}{1+c_{1}a_{2}^2} & \frac{-a_{3}a_{2}c_{2}}{1+c_{1}a_{2}^2} & 1\end{array}\right)$,
we obtain that $M_{B'}=\left(\begin{array}{cccccc}
1+a_{2}^2c_{1}+a_{3}^2c_{2} & 0 & 0\\[2mm]
c_{1}(1+c_{1}a_{2}^2) & 0 & 0\\[2mm]
\frac{c_{2}(1+a_{2}^2c_{1}+a_{3}^2c_{2})}{(1+c_{1}a_{2}^2)} & 0 & 0\end{array}\right).$\\[2mm]
\textbf{Case 1.2.6.2.1} Assume $1+a_{2}^2c_{1}+a_{3}^2c_{2}>0.$\\
Now, consider the natural basis $B''=\{f_{1},f_{2},f_{3}\}$
such that $$P_{B''B'}=\left(\begin{array}{cccccc}
\frac{1}{1+a_{2}^2c_{1}+a_{3}^2c_{2}} & 0 & 0\\[2mm]
0 & \frac{1}{\sqrt{c_{1}(1+c_{1}a_{2}^2)(1+a_{2}^2c_{1}+a_{3}^2c_{2})}} & 0\\[2mm]
0 & 0 & \frac{\sqrt{-(1+c_{1}a_{2}^2)}}{\sqrt{-c_{2}}(1+a_{2}^2c_{1}+a_{3}^2c_{2})}\end{array}\right)$$
and the structure matrix is:
$M_{B''}=\left(\begin{array}{cccccc}
1 & 0 & 0\\
1 & 0 & 0\\
1 & 0 & 0\end{array}\right),$ which has already appeared.\\[2mm]
\textbf{Case 1.2.6.2.2} Assume $1+a_{2}^2c_{1}+a_{3}^2c_{2}<0.$\\
Now, consider the natural basis $B''=\{f_{1},f_{2},f_{3}\}$
such that $$P_{B''B'}=\left(\begin{array}{cccccc}
\frac{1}{1+a_{2}^2c_{1}+a_{3}^2c_{2}} & 0 & 0\\[2mm]
0 & \frac{1}{\sqrt{-c_{1}(1+c_{1}a_{2}^2)(1+a_{2}^2c_{1}+a_{3}^2c_{2})}} & 0\\[2mm]
0 & 0 & \frac{\sqrt{-(1+c_{1}a_{2}^2)}}{\sqrt{-c_{2}}(1+a_{2}^2c_{1}+a_{3}^2c_{2})}\end{array}\right)$$
and the structure matrix is:
$$M_{B''}=\left(\begin{array}{cccccc}
1 & 0 & 0\\
-1 & 0 & 0\\
1 & 0 & 0\end{array}\right).$$\\[2mm]
\textbf{Case 1.2.7} Suppose that $c_{1}\neq0$, $c_{2}\neq0$ and $1+a_{2}^2c_{1}=0$.\\
Then $a_{2}a_{3}c_{1}c_{2}\neq0$ and so $c_{1}=-\frac{1}{a_{2}^2}$. For $B'$ we have
$$M_{B'}=\left(\begin{array}{cccccc}
a_{3}^2c_{2} & 0 & 0\\[2mm]
\frac{c_{2}}{a_{3}^2} & 0 & 0\\[2mm]
-a_{3}^2c_{2} & 0 & 0\end{array}\right).$$
Considering natural basis
$B''=\{\frac{1}{a_{3}^2c_{2}}e_{1},\frac{1}{c_{2}}e_{2},\frac{1}{a_{3}^2c_{2}}e_{3}\}$
we obtain
$M_{B''}=\left(\begin{array}{cccccc}
1 & 0 & 0\\
1 & 0 & 0\\
-1 & 0 & 0\end{array}\right)$. Which has already appeared.\\[2mm]
\textbf{Case 1.2.8} Suppose that $c_{1}\neq0$, and $c_{2}=0$.\\
Considering the natural basis $B''=\{e_{1},e_{3},e_{2}\}$ we obtain
$M_{B''}=\left(\begin{array}{cccccc}
1 & a_{3} & a_{2}\\
0 & 0 & 0\\
c_{1} & a_{3}c_{1} & a_{2}c_{1}\end{array}\right),$ and we
are in the same conditions as in Case 1.1.1.2.\\[2mm]
\textbf{Case 2} Suppose that $a_{1}=0.$\\
The structure matrix of the evolution algebra is
$$M_{B}=\left(\begin{array}{cccccc}
0 & a_{2} & a_{3}\\
0 & a_{2}c_{1} & a_{3}c_{1}\\
0 & a_{2}c_{2} & a_{3}c_{2}\end{array}\right).$$
Necessarily there exists $i\in\{2,3\}$ such that $a_{i}\neq0.$
Without loss in generality we assume $a_{2}\neq0.$\\[2mm]
\textbf{Case 2.1} Assume $c_{1}\neq0.$
Consider the natural basis $B''=\{e_{2},e_{3},e_{1}\}.$ Then \,\,\,\,\,\,\,\,\ $M_{B''}=\left(\begin{array}{cccccc}
a_{2}c_{1} & a_{3}c_{1} & 0\\
a_{2}c_{2} & a_{3}c_{2} & 0\\
1 & a_{3} & 0\end{array}\right)$
and we are in the same conditions as in Case 1.\\[2mm]
\textbf{Case 2.2} If $c_{1}=0.$\\[2mm]
\textbf{Case 2.2.1} Assume $c_{2}a_{3}\neq0.$\\
Taking the natural basis $B''=\{e_{3},e_{2},e_{1}\},$ we obtain  $M_{B''}=\left(\begin{array}{cccccc}
a_{3}c_{2} & a_{2}c_{2} & 0\\
0 & 0 & 0\\
a_{3} & a_{2} & 0\end{array}\right)$ and we are in the same conditions as in the Case1.\\[2mm]
\textbf{Case 2.2.2} Suppose that $c_{2}a_{3}=0.$\\[2mm]
\textbf{Case 2.2.2.1} Assume $c_{2}=0.$\\
Take the natural basis $B'=\{a_{2}e_{2}+a_{3}e_{3},\frac{1}{a_{2}}e_{3},e_{1}\}.$
Then $$M_{B'}=\left(\begin{array}{cccccc}
0 & 0 & 0\\
0 & 0 & 0\\
1 & 0 & 0\end{array}\right).$$
Note that $E^2$ has the extension property.\\[2mm]
\textbf{Case 2.2.2.2} Assume $c_{2}>0.$\\
Then $a_{3}=0.$ For $B'=\{a_{2}e_{2},e_{1},\frac{1}{\sqrt{c_{2}}}e_{3}\}$
we have $$M_{B'}=\left(\begin{array}{cccccc}
0 & 0 & 0\\
1 & 0 & 0\\
1 & 0 & 0\end{array}\right).$$
\textbf{Case 2.2.2.3} Assume $c_{2}<0.$\\
Then $a_{3}=0.$ For $B'=\{a_{2}e_{2},e_{1},\frac{1}{\sqrt{-c_{2}}}e_{3}\}$
we have $$M_{B'}=\left(\begin{array}{cccccc}
0 & 0 & 0\\
1 & 0 & 0\\
-1 & 0 & 0\end{array}\right).$$

Thus we have proved the following theorem.
\begin{theorem}\label{theorem3dim}
Any three dimensional real evolution algebra $E$ with $dim(E^2)=1$ is
isomorphic to one of the following pairwise non-isomorphic algebras:

$E_{1}: \left(\begin{array}{cccccc}
1 & 1 & 0\\
-1 & -1 & 0\\
0 & 0 & 0\end{array}\right)$,
$E_{2}: \left(\begin{array}{cccccc}
1 & 1 & 0\\
-1 & -1 & 0\\
1 & 1 & 0\end{array}\right)$,
$E_{3}: \left(\begin{array}{cccccc}
1 & 1 & 0\\
-1 & -1 & 0\\
-1 & -1 & 0\end{array}\right)$,
$E_{4}: \left(\begin{array}{cccccc}
1 & 0 & 0\\
0 & 0 & 0\\
0 & 0 & 0\end{array}\right)$,\\
$E_{5}: \left(\begin{array}{cccccc}
1 & 0 & 0\\
0 & 0 & 0\\
1 & 0 & 0\end{array}\right)$,
$E_{6}: \left(\begin{array}{cccccc}
1 & 0 & 0\\
0 & 0 & 0\\
-1 & 0 & 0\end{array}\right)$,
$E_{7}: \left(\begin{array}{cccccc}
1 & 0 & 0\\
1 & 0 & 0\\
1 & 0 & 0\end{array}\right)$,\\
$E_{8}: \left(\begin{array}{cccccc}
1 & 0 & 0\\
1 & 0 & 0\\
-1 & 0 & 0\end{array}\right)$,
$E_{9}: \left(\begin{array}{cccccc}
1 & 0 & 0\\
-1 & 0 & 0\\
-1 & 0 & 0\end{array}\right)$,
$E_{10}: \left(\begin{array}{cccccc}
1 & 0 & 0\\
-1 & 0 & 0\\
1 & 0 & 0\end{array}\right)$,\\
$E_{11}: \left(\begin{array}{cccccc}
0 & 0 & 0\\
0 & 0 & 0\\
1 & 0 & 0\end{array}\right)$,
$E_{12}: \left(\begin{array}{cccccc}
0 & 0 & 0\\
1 & 0 & 0\\
1 & 0 & 0\end{array}\right)$,
$E_{13}: \left(\begin{array}{cccccc}
0 & 0 & 0\\
1 & 0 & 0\\
-1 & 0 & 0\end{array}\right).$
\end{theorem}

\begin{remark}
One can classify real three-dimensional evolution algebras in case $dim(E^2)\neq1$.
But it will contain too long cases and subcases.
\end{remark}

\section{Approximation of three-dimensional evolution algebras
($dim(E^2)=1$)}\label{section:5}

In this section for the
evolution algebras $E_{i}, i=\overline{1,13}$ from
Theorem \ref{theorem3dim} we will construct evolution
algebras corresponding to fixed points of the operator $F$.

Let $E$ be three dimensional evolution algebra with the matrix
$(a_{ij}), i,j\in\{1,2,3\}$. We will rewrite the operator
$F$ for this evolution algebra as:
\begin{displaymath}
F: \left\{ \begin{array}{ll}
x_{1}'=a_{11}x_{1}^2+a_{21}x_{2}^2+a_{31}x_{3}^2, \\[2mm]
x_{2}'=a_{12}x_{1}^2+a_{22}x_{2}^2+a_{32}x_{3}^2, \\[2mm]
x_{3}'=a_{13}x_{1}^2+a_{23}x_{2}^2+a_{33}x_{3}^2.
\end{array} \right.
\end{displaymath}

Jacobian of the operator $F$ at the point $x$ has a form
\begin{displaymath}
J_{F}(x)=\left( \begin{array}{cccccc}
2a_{11}x_{1} & 2a_{21}x_{2} & 2a_{31}x_{3} \\
2a_{12}x_{1} & 2a_{22}x_{2} & 2a_{32}x_{3} \\
2a_{13}x_{1} & 2a_{23}x_{2} & 2a_{33}x_{3}
\end{array} \right).
\end{displaymath}

Following \cite{9} and \cite{3} we define an evolution algebra $\widetilde E$ with matrix
$J_{F}(x)$ as the matrix of structural constants.

There is no non-zero fixed point of the operator $F$ for the
evolution algebras  $E_{i},$  $i\in\{1,2,3,$ $11,12,13\}$ and
$(1;0;0)$ is the unique fixed point of the operator $F$ for the
evolution algebras $E_{i}, i=\overline{4,10}$. So
$$J_{F}(1;0;0)=\left( \begin{array}{cccccc}
2 & 0 & 0 \\
0 & 0 & 0 \\
0 & 0 & 0
\end{array} \right).$$
It is easy to see that the evolution algebra with the matrix $J_{F}(1;0;0)$
is isomorphic to the evolution algebra $E_{4}$.

\section{Acknowledgement}

The author is indebted to Professors Sh.Ayupov and U.Rozikov for 
the useful advice and valuable comments.

The author thanks referees for their helpful comments.

\bigskip
\begin{figure}[ht!]
\noindent\rule{\textwidth}{0.5pt}\\[-.8\baselineskip]\rule{\textwidth}{0.5pt}\\
\begin{minipage}[b]{0.5\linewidth}
\vspace{0.1cm}
\qquad\qquad\qquad\epsfig{file=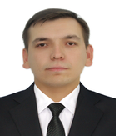,height=1.71in,width=1.45in}
\end{minipage}
\begin{minipage}[b]{7cm}
\bigskip \bigskip\footnotesize{\textbf{ Anvar Imomkulov}
         is a PhD student of Institute of Mathematics, Tashkent, Uzbekistan.
        He graduated from National University of Uzbekistan (2010).}
        \end{minipage}\\
\end{figure}


\begin{thebibliography}{99}

\bibitem{1}\label{Cabrera} Cabrera Casado, Yolanda; Siles Molina, Mercedes;
Velasco, M. Victoria, (2017), Classification of three-dimensional evolution
 algebras, Linear Algebra Appl. {\bf 524}, pp.68-108.

\bibitem{2} \label{Djumadildayev} Dzhumadil'daev, A.; Omirov, B. A.;
Rozikov, U. A., (2014), On a class of evolution algebras of "chicken''
population, Internat. J. Math. {\bf 25}, no. 8, 1450073, 19 pp.

\bibitem{3}\label{Imomkulov} Imomkulov A.N., (2018), Isomorphicity of two dimensional
evolution algebras and evolution algebras corresponding
to their idempotents,  Uzb.Math.Jour. {\bf 2} pp.63-73.

\bibitem{4} \label{ladra} Ladra M., Rozikov U.A., (2013), Evolution algebra
of a bisexual population, Jour.Alg. {\bf 378} pp.153-172.

\bibitem{5} \label{Lyubich} Lyubich Y.I., (1992), Mathematical structures in
population genetics, Springer-Verlag, Berlin,.

\bibitem{6} \label{Murodov} Murodov Sh.N., (2014), Classification of two-dimensional
real evolution algebras and dynamics of some two-dimensional chains
of evolution algebras, Uzb.Mat.Jour. {\bf 2} pp.102-111.

\bibitem{7} \label{MurodovInter} Murodov Sh.N., (2014), Classification dynamics of
two-dimensional chains of evolution algebras,
Internat. J. Math. {\bf 25}, no. 2, 1450012, 23 pp.

\bibitem{8} \label{Rozikov} Rozikov U.A., Tian J.P., (2011), Evolution algebras
generated by Gibbs measures, Lobachevskii Jour. Math. {\bf 32} (4) pp.270-277.

\bibitem{9} \label{Velasco} Rozikov U.A., Velasco M.V., A discrete-time
dynamical system and an evolution algebra of mosquito population,
 J.Math.Biology. DOI 10.1007/500285-018-1307-x

\bibitem{10} \label{tian} Tian J.P., (2008), Evolution Algebras and Their Applications,
 Lecture Notes in Math., 1921, (Springer, Berlin,.), 125 p.

\end{thebibliography}
\end{document}